\documentclass[10pt, reqno]{amsart}
\usepackage{amscd, amssymb,latexsym,amsmath, amscd, amsmath, comment, hyperref}
\usepackage[all]{xy}
\usepackage{anysize}
\usepackage[sc]{mathpazo} 
\usepackage{hyperref}
\usepackage{bbm}
\usepackage{color}
\usepackage{dsfont}
\usepackage{mathtools}
\usepackage{cases}
\usepackage[numbers]{natbib}

\marginsize{2.0cm}{2.0cm}{2.0cm}{1.9cm}
\usepackage{mathtools}
\usepackage{dsfont}

\def\F{{\mathbb F}}
\def\OF{{\overline{\mathbb F}}_p}
\def\OFs{{\overline{\mathbb F}}_p^2}
\def\Gq{{G_q}}
\def\Gp{{G_p}}
\def\CL{{\chi_{_{\mathrm{LHS}}}}}
\def\CR{{\chi_{_{\mathrm{RHS}}}}}
\def\xtx{{X\otimes X}}
\def\xty{{X\otimes Y}}
\def\ytx{{Y\otimes X}}
\def\yty{{Y\otimes Y}}

\numberwithin{equation}{section}

\newtheorem{theorem}{Theorem}[section] 
\newtheorem{lemma}[theorem]{Lemma}
\newtheorem{corollary}[theorem]{Corollary}
\newtheorem{proposition}[theorem]{Proposition}
\theoremstyle{remark}
\newtheorem{remark}{Remark}

\newtheorem{definition}{Definition}

\begin{document}

        \title[Restriction of irreducible representations]{On the restriction of some irreducible mod $p$ representations}
	\author{Shubhanshi Gupta}
	
        \address{Department of Basic Sciences, IITRAM, Ahmedabad - 380026, India.}
	\email{shubhanshi.gupta.20pm@iitram.ac.in}

	\subjclass{Primary 20C20}
        \keywords{Modular representations, tensor product decomposition, Clebsch-Gordan problem, distinguished representations}
	\date{}

        \begin{abstract}
          For a prime $p,$ let $\F_q$ be a finite extension of $\F_p.$ The restriction of an irreducible mod $p$ representation of $\text{GL}_2(\F_q)$ to its subgroup $\text{GL}_2(\F_p)$ can be seen as a tensor product of irreducible representations of $\text{GL}_2(\F_p).$ In this paper, we study the restriction of some of these representations of $\text{GL}_2(\F_q)$ to $\text{GL}_2(\F_p),$ for $q=p^2$ and $p^3$ using elementary tools and give explicit socle filtration when $q=4.$ We prove that when $q=p^2,$ a special class of representations of $\text{GL}_2(\F_q)$ are distinguished by suitable characters of $\text{GL}_2(\F_p).$
	\end{abstract}
	\maketitle

\section{Introduction}\label{S1}

   For a representation $\rho$ of a group $G$ and its subgroup $H,$ it is interesting to know the structure of the restriction ${\rho|}_H.$ Depending on various factors, such as the groups involved and the underlying field of the representation space, the restriction of an irreducible representation may remain irreducible, become completely reducible, or become indecomposably reducible. In particular, for finite groups, representations over fields with characteristic zero are completely reducible due to Maschke's theorem. This is not the case in positive characteristic. In this article, our primary interest is in the restriction of an irreducible mod $p$ representation of $\text{GL}_2(\F_q)$ to $\text{GL}_2(\F_p),$ for a finite extension $\F_q$ of a finite field $\F_p.$ The irreducible representations of $\text{GL}_2(\F_q)$ are classified as $\bigotimes\limits_{i=1}^n\text{Sym}^{r_i}(\OFs)\otimes\mathrm{det}^k$ with $0\leq r_i<p$ and $0\leq k< q-1,$ where $q=p^n$ and $\text{Sym}^j(\OFs)$ denotes the $j^{th}$ symmetric power of the standard 2-dimensional representation. The action involves increasing Frobenius powers in each component of the tensor product and since $\alpha^{p^i}=\alpha$ for $\alpha\in\F_p,$ its restriction to $\text{GL}_2(\F_p)$ can be seen as the tensor product of representations of $\text{GL}_2(\F_p).$
   
In Section \ref{S3}, we analyze the restriction of irreducible representations of $\text{GL}_2(\F_{p^2})$ to $\text{GL}_2(\F_p).$ The decomposition when $0\leq r_1+r_2< p$ can be given by the Clebsch-Gordan decomposition (cf. 5.5(a) in \cite{DG}) which requires proving the exactness of the sequence 5.1 in \cite{DG}. Our approach is more computational yet simpler than this. In Section \ref{S3} of this paper, we derive the decomposition up to semisimplification using only the ordinary character theory, instead of involving the above exact sequence as a major tool. Then we use Lemma \ref{L1} to prove the actual decomposition. We also give an explicit structure of the restriction of the irreducible representations of $\text{GL}_2(\F_4)$ to $\text{GL}_2(\F_2)$ in Subsection \ref{SS1}.

Another concept of general interest associated with restriction is that of distinction. Such problems are not as elaborately studied in mod $p$ as in complex or $\ell\text{-modular}$ representation theory; see, for instance, articles like \cite{KC}, \cite{CLL} and \cite{VS} and the references therein. In Subsection \ref{SS2}, as a consequence of Theorem \ref{T2}, we give the necessary and sufficient conditions for the irreducible representations of $\mathrm{GL}_2(\F_{p^2})$ (when $0\leq r_1,r_2<p$ and $r_1+r_2<p$) to be $\text{GL}_2(\F_p)\text{-distinguished}.$ We also show that the character $\mathrm{det}^{r+k}$ appears in the socle as well as in the cosocle of the representation $\text{Sym}^{r}\otimes\text{Sym}^{r}\otimes\mathrm{det}^k$ of $\text{GL}_2(\F_p)$ for all $r,k\geq 0.$ Consequently, as representations of $\mathrm{GL}_2(\F_{p^2}),$ such representations are $(\text{GL}_2(\F_p),\mathrm{det}^{r+k})\text{-distinguished}$ (cf. Definition \ref{Def}).

Section \ref{S4} contains one of the main results of this paper. It deals with the restriction of irreducible representations of $\text{GL}_2(\F_{p^3})$ to $\text{GL}_2(\F_p)$ when $r_1+r_2+r_3<p$ as proved in Theorem \ref{T3}. Here, we not only give sufficient condition for the representation to be completely reducible but also explicitly find the irreducible components with multiplicities. These sufficient conditions are in compliance with the result of J.P. Serre in \cite{JPS} which proves that the tensor product of $m$ finite dimensional irreducible representations $\{V_i: 1\leq i\leq m\}$ of a group is semisimple whenever $\sum\limits_{i=1}^m(\text{dim }V_i-1)<p.$ This motivates us to look for appropriate generalizations for any finite extension $\F_q$ of $\F_p,$ which could potentially be our next work. The idea is to use the decomposition of the tensor product of two irreducible representations and the distributivity of tensor product over the direct sum. A similar result is proved for the restriction of irreducible representations of $\text{SL}_2(\F_{p^3})$ to $\text{SL}_2(\F_p)$ as a corollary.

	\section{Preliminary}\label{S2}
	We fix an algebraic closure ${\OF}$ of the finite field $\F_p.$ For a finite extension $\F_q$ of $\F_p,$ let $\Gp\coloneqq\mathrm{GL}_2(\F_p)$ and $\Gq\coloneqq\mathrm{GL}_2(\F_q).$
    
\subsection{Character theory}
	
	The character theory in the study of representations over fields of characteristic $p$ is inadequate. For instance, unlike in complex representation theory, the equality of the characters of two representations of a finite group does not guarantee isomorphic representations. However, a weaker result holds here for finite groups, which is formulated in the next theorem.

\begin{theorem}\label{T5}
		Let $\chi_1$ and $\chi_2$ be the ordinary characters of two representations $V_1$ and $V_2$ respectively over $\OF.$ If $\chi_1=\chi_2$ then for each composition factor $W$ occurring in $V_1$ or $V_2,$ the multiplicity of $W$ in $V_1$ is congruent modulo $p$ to the multiplicity of $W$ in $V_2.$
	\end{theorem}
	
	\begin{proof}
		Refer to the proof of Theorem $7.2$ in \cite{DB}.
	\end{proof}
	\begin{corollary}\label{C}
		If $\{W_i: 1\leq i\leq k\}$ is a set of distinct composition factors of a representation $V$ such that characters of $V$ and $\bigoplus\limits_{i=1}^k W_i$ are equal, then	$V^{\mathrm{ss}}\cong\bigoplus\limits_{i=1}^k W_i$ if and only if $\text{dim}(V)=\sum\limits_{i=1}^k\text{dim}(W_i),$ where $V^{\textrm{ss}}$ denotes the semisimplification of $V.$
 
	\end{corollary}
    
	\begin{proof}
		Without loss of generality, assume that $V$ is semisimple. The characters of $V$ and $\bigoplus\limits_{i=1}^k W_i$ being equal implies by Theorem \ref{T5} that each $W_i$ occurs in $V$ at least once i.e., $\bigoplus\limits_{i=1}^k W_i\hookrightarrow V.$ Thus, $\text{dim}(V)=\text{dim}\left(\bigoplus\limits_{i=1}^k W_i\right)$ if and only if  $V\cong\bigoplus\limits_{i=1}^k W_i.$
	\end{proof}

    \subsection{Restriction of irreducible representations of $\Gq$ to $\Gp$}
    Let $\F_q$ be a finite extension of $\F_p$ such that $q=p^n.$
	A complete list of the irreducible mod $p$ representations of $\Gq$ (cf. \cite{BL}) is given by 
    $$V_{\vec{r}}(k)= \text{Sym}^{r_1}(\OFs)\otimes\text{Sym}^{r_2}(\OFs)\otimes\cdots\otimes\text{Sym}^{r_n}(\OFs)\otimes \mathrm{det}^k,$$
	where $\Vec{r}=(r_1,r_2,\cdots, r_n)$ with $r_i\in \{0,1,\cdots, p-1\}$ and $k\in\{0,1,\cdots, q-2\}$ and where $\text{Sym}^{r_i}(\OFs)$ is the $r_i+1$ dimensional vector space of homogeneous polynomials over $\OF$ in two variables $x_i$ and $y_i$ of degree $r_i,$ that is, $\text{Sym}^{r_i}(\OFs)=\left\langle\left\{x_i^{r_i-j}y_i^j: 0\leq j\leq r_i \right\} \right\rangle_{\OF}.$\\
    The action of $\left(
	\begin{matrix}
		a & b\\
		c & d
	\end{matrix}
	\right) \in\Gq$ on $\bigotimes\limits_{i=1}^n P_i(x_i,y_i)\in V_{\vec{r}}(k)$ is defined as 

				$$\left(
		\begin{matrix}
			a & b\\
			c & d
		\end{matrix}
		\right)\cdot \left(\bigotimes\limits_{i=1}^n P_i(x_i,y_i) \right)= \bigotimes\limits_{i=1}^n\left(\left(
		\begin{matrix}
			a^{p^{i-1}} & b^{p^{i-1}}\\
			c^{p^{i-1}} & d^{p^{i-1}}
		\end{matrix}
		\right)\cdot P_i(x_i,y_i)\right)\otimes \mathrm{det}^k\left(
		\begin{matrix}
			a & b\\
			c & d
		\end{matrix}
		\right),$$ where $\left(
	\begin{matrix}
		a & b\\
		c & d
	\end{matrix}
	\right)\cdot P_i(x_i,y_i)=P_i
    \left(ax_i+cy_i,bx_i+dy_i\right).$
    
\begin{remark}
    We will use the following notations throughout: $$V_{\vec{r}}\coloneqq V_{\vec{r}}(0)\text{ and }V_r(k)\coloneqq \text{Sym}^{r}(\OFs)\otimes\mathrm{det}^k.$$ 
\end{remark}
   
Note that since $\alpha^{p^i}=\alpha$ for all $\alpha\in\F_p,$ the restriction of $V_{\vec{r}}(k)$ to $\Gp$ has the following action
$$\left(
		\begin{matrix}
			a & b\\
			c & d
		\end{matrix}
		\right)\cdot \left(\bigotimes\limits_{i=1}^n v_i\right)=\bigotimes\limits_{i=1}^n\left(\left(
		\begin{matrix}
			a & b\\
			c & d
		\end{matrix}
		\right)\cdot v_i\right)\otimes \mathrm{det}^{k}\left(
		\begin{matrix}
			a & b\\
			c & d
		\end{matrix}
		\right).$$
Thus, the restriction of irreducible representations of $\Gq$ to $\Gp$ becomes the usual tensor product of finitely many irreducible representations of $\Gp.$

	\section{Restriction when $\F_q$ is a quadratic extension of $\F_p$}\label{S3}

	For this section, let $\F_q$ be a quadratic extension of $\F_p.$ Here, the restriction of $V_{\vec{r}}(k)$ to $\Gp$  is equivalent to studying the representation $V_{r_1}\otimes V_{r_2}\otimes\mathrm{det}^k$ of $\Gp.$ For brevity, we assume that $k=0$ and then remark later the result for any $k.$ The tensor product of two irreducible representations is usually known as the Clebsch-Gordan problem \cite{FH}. The contemplated result is as follows.
	\begin{theorem}\label{T1}
		For $\vec{r}=(r_1,r_2)$ with $0\leq r_1, r_2< p$ and $r_1+r_2<p,$
		\begin{equation}\label{E1}
			\tag{I}
			\left({V_{\vec{r}}|}_{\Gp}\right)^{\mathrm{ss}}\cong\left(V_{r_1}\otimes V_{r_2}\right)^\mathrm{ss}\cong V_{r_1+r_2}\oplus V_{r_1+r_2-2}(1)\oplus\cdots\oplus V_{|r_1-r_2|}(\mathrm{min}\{r_1,r_2\}).
		\end{equation}
	\end{theorem}
	We prove this through a series of results on characters of these representations. For that, recall that there are four types of conjugacy classes of $\Gq$, classified by the rational canonical form, namely, $\left(
			\begin{matrix}
				\lambda_1 & 0\\
				0 & \lambda_1
			\end{matrix}
			\right) , \left(
			\begin{matrix}
				\lambda_1 & 1\\
				0 & \lambda_1
			\end{matrix}
			\right)$ and $\left(
			\begin{matrix}
				\lambda_1 & 0\\
				0 & \lambda_2
			\end{matrix}
			\right)$ with $\lambda_1\neq\lambda_2$ and $\left(
			\begin{matrix}
				0 & -b\\
				1 & -a
			\end{matrix}
			\right),$ where $\lambda_1,\lambda_2,b\in\F_q^*$ and $a\in\F_q$. Without loss of generality, assume that $r_1\geq r_2.$ Henceforth, we will denote the character of $V_{r_1}\otimes V_{r_2}$ by $\CL$ and the character of $V_{r_1+r_2}\oplus V_{r_1+r_2-2}(1)\oplus\cdots\oplus V_{r_1-r_2}(r_2)$ by $\CR.$
	
	\begin{lemma}\label{L1}
		Let $g=\left(
		\begin{matrix}
			a & 0\\
			0 & d
		\end{matrix}
		\right)\in \Gp.$ Then, $\CL(g)=\CR(g).$
	\end{lemma}
	
	\begin{proof}
		The character of $V_r(k)$ on $g$ is $\left(\sum\limits_{i=0}^{r}a^{r-i}d^i\right)\left(ad\right)^k.$ Therefore,
		\begin{itemize}
			\item[] when $a=d,$
			\begin{align*}
				\CR(g)&= a^{r_1+r_2}[(r_1+r_2+1)+(r_1+r_2-1)+\cdots+(r_1-r_2+1)]\\
				&= a^{r_1+r_2}[(r_2+1)(r_1+r_2)+1-1-3-5-\cdots-(2r_2-1)]\\
				&= a^{r_1+r_2}(r_1+1)(r_2+1)\\
				&=\CL(g)
			\end{align*}
			
			\item[] and when $a\neq d,$
			\begin{align*}				
				\CL(g)&=\left(a^{r_1}+a^{{r_1}-1}d+\cdots + d^{r_1}\right)\left(a^{r_2}+a^{{r_2}-1}d+\cdots +d^{r_2}\right)\\				
				&= \dfrac{(a^{r_1+1}-d^{r_1+1})}{(a-d)}\left(a^{r_2}+a^{{r_2}-1}d+\cdots +d^{r_2}\right)\\				&=\dfrac{\left(a^{r_1+r_2+1}+a^{r_1+r_2}d+\cdots+a^{r_1+2}d^{r_2-1}+a^{r_1+1}d^{r_2}\right)}{(a-d)}\\
				&-~\dfrac{\left(a^{r_2}d^{r_1+1}+a^{r_2-1}d^{r_1+2}+\cdots+ad^{r_1+r_2}+d^{r_1+r_2+1}\right)}{(a-d)}\\
				&=\dfrac{a^{r_1+r_2+1}-d^{r_1+r_2+1}}{(a-d)}+\dfrac{ad[a^{r_1+r_2-1}-d^{r_1+r_2-1}]}{(a-d)}+\cdots+\dfrac{(ad)^{r_2}[a^{r_1-r_2+1}-d^{r_1-r_2+1}]}{(a-d)}\\
				&= \left(\chi_{_{V_{{r_1}+{r_2}}}}+
				\chi_{_{V_{{r_1}+{r_2}-2}(1)}}+\cdots+\chi_{_{V_{{r_1}-{r_2}+2}({r_2}-1)}}+\chi_{_{V_{{r_1}-{r_2}}(r_2)}}\right)(g)\\
				&=\CR(g).
			\end{align*} \end{itemize}
	\end{proof}

	\begin{corollary}\label{C1}
		Let $g=\left(
		\begin{matrix}
			a & 1\\
			0 & d
		\end{matrix}
		\right)\in \Gp.$ Then, $\CL(g)=\CR(g).$
	\end{corollary}
	
	\begin{proof}
		Let $g=\left(
		\begin{matrix}
			a & 1\\
			0 & d
		\end{matrix}
		\right)\in \Gp$ and $\text{Sym}^r(\OFs)$ be a representation of $\Gp.$ Then with respect to the standard basis, note that $g\longmapsto \left(\begin{matrix}
			a^r &*&\cdots&*\\
			0 &a^{r-1}d&\cdots& *\\
			\vdots&\vdots&\ddots&\vdots\\
			0&0&\cdots&d^r
		\end{matrix}
		\right).$ Thus, $\chi_{_{V_r}}\left(\begin{matrix}
			a & 1\\
			0 & d
		\end{matrix}
		\right)=\chi_{_{V_r}}\left(
		\begin{matrix}
			a & 0\\
			0 & d
		\end{matrix}
		\right).$ Using similar arguments as in Lemma \ref{L1}, we have that $\CL(g)=\CR(g).$
	\end{proof}
	
	\begin{corollary}\label{C2}
		If the characteristic polynomial of $g\in \Gp$ is irreducible over $\F_p$ then, $\CL(g)=\CR(g).$
	\end{corollary}
	
	\begin{proof}
		Let $x^2+ax+b$ be the characteristic polynomial of $g$ and $\alpha\in\overline{\F}_p$ be its root. Then the rational canonical form of $g$ is given by $\left(
		\begin{matrix}
			0 & -b\\
			1 & -a
		\end{matrix}
		\right)$ which is conjugate to $\left(
		\begin{matrix}
			\alpha & 0\\
			0 & \overline\alpha
		\end{matrix}
		\right)\text{ in } \text{GL}_2(\F_p[\alpha]).$ Using similar arguments as in Lemma \ref{L1}, the result holds.
	\end{proof}
	\begin{proof}[Proof of Theorem \ref{T1}] We have shown that the characters $\CL$ and $\CR$ agree on the representatives of all the conjugacy classes of $\Gp$ (Lemma \ref{L1}, Corollary \ref{C1} and Corollary \ref{C2}). Since character is a class function, we get $\CL=\CR$ on $\Gp.$
		
		Note that in Theorem \ref{T1}, the dimension of LHS in (\ref{E1}) is $(r_1+1)(r_2+1),$ while the dimension of RHS in (\ref{E1}) is		
		\begin{align*}
			\text{dim(RHS)}&=(r_1+r_2+1)+(r_1+r_2-2+1)+\cdots+ (r_1-r_2+1)\\
			&=(r_1+\mathbf{r_2}+1)+\cdots+\left(r_1+\mathbf{r_2-(2r_2-2)}+1\right)+\left(r_1+\mathbf{r_2-(2r_2)}+1\right)\\
			&=(r_2+1)(r_1+1)+\left[\mathbf{r_2+(r_2-2)+\cdots+(r_2-(2r_2-2))+(r_2-2r_2)}\right]\\
			&=(r_2+1)(r_1+1)+(r_2+1)r_2-2\left(\dfrac{r_2(r_2+1)}{2}\right)\\
			&= (r_2+1)(r_1+1)= \text{dim(LHS)}.
		\end{align*}
		Since $r_1+r_2<p,$ all the components on the RHS of (\ref{E1}) are irreducible representations of $\Gp.$ Thus, the theorem follows from Corollary \ref{C}.
	\end{proof}

	The conclusion of Theorem \ref{T1} is weaker in the sense that the isomorphism is only up to semisimplification and $p$ has to be sufficiently large. We make this result stronger using the following lemma. This lemma is result 5.2 in \cite{DG}, but it is proved by using the exactness of sequence 5.1 in the same article, while we use our result \ref{T1} specifically to do the same.
	\begin{lemma}\label{L2}
		For $1\leq r< p-1,$  there is a $\Gp\text{-isomorphism}$ $$V_r\otimes V_1\cong V_{r+1}\oplus V_{r-1}(1).$$
	\end{lemma}
	
	\begin{proof}
		Let $1\leq r<p-1.$ Then we have $(V_r\otimes V_1)^{\text{ss}}\cong V_{r+1}\oplus V_{r-1}(1),$ by Theorem \ref{T1}. Thus, it is sufficient to construct onto $\Gp\text{-linear}$ maps from $V_r\otimes V_1$ to each of $V_{r+1}$ and $V_{r-1}(1).$
		
		Define $P_1:V_r\otimes V_1\to V_{r+1}$ by the linear extension of $f_1\otimes f_2\mapsto f_1f_2.$ One can easily verify that $P_1$ is an onto $\Gp\text{-linear}$ map. Let $\{X,Y\}$ be a basis of $V_1$. Then, define \\
		$P_2: V_r\otimes V_1\to V_{r-1}(1)$ by the linear extension of the map, which sends $$f\otimes X\mapsto \dfrac{\partial{f}}{\partial{Y}}\text{ and }f\otimes Y\mapsto -\dfrac{\partial{f}}{\partial{X}}.$$ Note that $P_2$ is onto: $P_2\left(\dfrac{X^{r-1-i}Y^{i+1}}{i+1}\otimes X\right)=X^{r-1-i}Y^{i}, \text{ for all }i=0,1,\cdots,r-1.$ $\Gp\text{-linearity}$ of $P_2$ can be checked using the properties of formal partial derivatives.
	\end{proof}
	
	\begin{remark}\label{R1}
		Since $V_{r_1}\otimes V_{r_2}$ and $V_{r_2}\otimes V_{r_1}$ are isomorphic, it is enough to treat the cases where $1\leq r_2\leq r_1< p.$
	\end{remark}
	
	\begin{remark}\label{R2}
		As $\Gp\text{-representations, } V_{r_1}\otimes V_0\cong V_{r_1}$ and $V_0\otimes V_{r_2}\cong V_{r_2}$ trivially.
	\end{remark}
	
	\begin{theorem}\label{T2}
		Let $\vec{r}=(r_1,r_2)$ with $0\leq r_2\leq r_1< p-r_2.$ Then $${V_{\vec{r}}|}_{\Gp}\cong V_{r_1+r_2}\oplus V_{r_1+r_2-2}(1)\oplus \cdots\oplus V_{r_1-r_2}(r_2).$$
	\end{theorem}
	
	\begin{proof}
		We will prove the result by induction on $r_2.$ When $r_2=0,$ the theorem holds by Remark \ref{R2}.
		Let $r_2=1,$ then for $1\leq r_1<p-1,$ the result holds in this case by Lemma \ref{L2}.
		
		As $0\leq r_2\leq r_1<p-r_2,$ all the components in the decomposition are irreducible. So, we have $r_2\leq \frac{p-1}{2}.$ Let the induction hypothesis be that for all $r_2\in\left\{1,2,\cdots,\frac{p-1}{2}-1\right\}$ with $r_2\leq r_1<p-r_2,$ 
		$$V_{r_1}\otimes V_{r_2}\cong V_{r_1+r_2}\oplus V_{r_1+r_2-2}(1)\oplus\cdots\oplus V_{r_1-r_2}(r_2).$$
		
		We have $V_{r_1}\otimes (V_1\otimes V_{r_2})\cong (V_{r_1}\otimes V_1)\otimes V_{r_2}.$
		Since $1\leq r_2\leq r_1<p-r_2\leq p-1,$ by Lemma \ref{L2} we get $V_{r_1}\otimes (V_{r_2+1}\oplus V_{r_2-1}(1))\cong (V_{r_1+1}\oplus V_{r_1-1}(1))\otimes V_{r_2},$ that is $$(V_{r_1}\otimes V_{r_2+1})\oplus (V_{r_1}\otimes V_{r_2-1}(1))\cong (V_{r_1+1}\otimes V_{r_2})\oplus (V_{r_1-1}(1)\otimes V_{r_2}).$$
		Further, using induction hypothesis, for $r_2\leq r_1-1<p-r_2$ and $r_2\leq r_1+1<p-r_2,$ we have
		\begin{multline*}
			(V_{r_1}\otimes V_{r_2+1})\oplus [V_{r_1+r_2-1}(1)\oplus V_{r_1+r_2-3}(2)\oplus\cdots\oplus V_{r_1-r_2+1}(r_2)]\cong\\
            [V_{r_1+r_2+1}\oplus V_{r_1+r_2-1}(1)\oplus \cdots \oplus V_{r_1-r_2+1}(r_2)]\oplus[V_{r_1+r_2-1}(1)\oplus V_{r_1+r_2-3}(2)\oplus \cdots \oplus V_{r_1-r_2-1}(r_2+1)].
		\end{multline*}
		
		And so, for $r_2+1\leq r_1<p-r_2-1,$ we get
		$$V_{r_1}\otimes V_{r_2+1}\cong V_{r_1+r_2+1}\oplus V_{r_1+r_2-1}(1)\oplus \cdots\oplus V_{r_1-r_2-1}(r_2+1).$$
	\end{proof}

\begin{remark}\label{R3}
	By taking twists by $\mathrm{det}^k$ for $k\in \{0,1,\cdots, q-2\}$ in the above theorem, we get $${V_{\vec{r}}(k)|}_{\Gp}\cong V_{r_1+r_2}(k)\oplus V_{r_1+r_2-2}((1+k))\oplus\cdots\oplus V_{r_1-r_2}((r_2+k)).$$
	\end{remark}
	
	\begin{corollary}
		Let $0\leq r_2\leq r_1<p-r_2,$ then the restriction of $V_{\vec{r}}$ as an irreducible representation of $\text{SL}_2(\F_q)$ to $\text{SL}_2(\F_p)$ is $V_{r_1+r_2}\oplus V_{r_1+r_2-2}\oplus\cdots\oplus V_{r_1-r_2}.$
	\end{corollary}
	\begin{proof}
		The irreducible representations of $\text{SL}_2(\F_q)$ up to isomorphism are $V_{r_1}\otimes V_{r_2},$ where $r_1,r_2\in\{0,1,\cdots, p-1\}$ with respect to the same action as for $\Gq.$ Now, $\Gp\text{-isomorphism}$ in Theorem \ref{T2} implies $\text{SL}_2(\F_p)\text{-isomorphism}$ and so 
		\begin{align*}
			{V_{\vec{r}}|}_{\text{SL}_2(\F_p)}=V_{r_1}\otimes V_{r_2}&\cong V_{r_1+r_2}\oplus V_{r_1+r_2-2}(1)\oplus \cdots\oplus V_{r_1-r_2}(r_2)\\
			&\cong V_{r_1+r_2}\oplus V_{r_1+r_2-2}\oplus\cdots\oplus V_{r_1-r_2}.        
		\end{align*}
	\end{proof}
	
	\subsection{Complete structure for prime $p=2$}\label{SS1}
	For this subsection, let $p=2$ and $q=4.$  The irreducible representations of $G_2$ are $V_{\vec{r}}$ such that $0\leq r_1,r_2<2.$ Remark \ref{R2} treats the restriction of all but one of these namely $V_1\otimes V_1,$ which we shall explicitly study. Note that the only irreducible representations of $G_2$ are $V_0$ and $V_1,$ that is trivial ($\mathds{1}$) and standard (std) representations respectively.
	
	\begin{lemma}
		Let $V=V_1\otimes V_1.$ Then $\mathds{1}\oplus\mathrm{std}\subseteq V$ as $G_2\text{-representations}$.
	\end{lemma}
	\begin{proof}
		Let $\left\langle\{\xtx, \xty,\ytx,\yty\}\right\rangle_{\overline{\F}_2}$ be a basis of $V$ and $g=\left(\begin{matrix}
			a & b \\
			c & d
		\end{matrix}\right)\in G_2.$ Then $\mathds{1}=\langle\{\xty-\ytx\}\rangle_{\overline{\F}_2}$ is a trivial subrepresentation of $V.$ This can be verified easily since the characteristic of the coefficient field is $2$ and so, $ad-bc=1.$ Similarly, $\mathrm{std}=\langle\{\xtx+\xty-\ytx, ~\yty+\xty-\ytx\}\rangle_{\overline{\F}_2}$ is a standard subrepresentation of $V.$ Here, $G_2\text{-linearity}$ can easily be checked on the basis elements, for instance:
		 \begin{align*}
            g\cdot(\xtx+\xty-\ytx)=&  (aX+cY)\otimes(aX+cY)+(aX+cY)\otimes(bX+dY)-(bX+dY)\otimes(aX+cY)\\
				=& (a^2+ab-ab)(\xtx)+(ac+ad-bc)(\xty)+\\
				& (ac+bc-ad)(\ytx)+(c^2+cd-cd)(\yty)\\
				=& a(\xtx)+(a+c)(\xty-\ytx)+c(\yty)\\
				&(\because ac+1=a+c, \text{ for all } g\in G_2)\\
				=& a(\xtx+\xty-\ytx)+c(\yty+\xty-\ytx).
			\end{align*}

		Clearly, $\mathds{1}\cap \mathrm{std}=\{0\}$ and hence, proved.
	\end{proof}
	
	\begin{proposition}
		The socle filtration of $V$ is $0\subseteq \mathds{1}\oplus \mathrm{std}\subseteq V$ (with respect to the notations of the above lemma).
	\end{proposition}
	
	\begin{proof}
		Let $W=\mathds{1}\oplus \mathrm{std}.$ Consider the short exact sequence
		$$0\longrightarrow W\overset{f}{\longrightarrow}V\overset{}{\longrightarrow}V/W\longrightarrow 0$$ where $W=\langle\{1,x,y\}\rangle,$ such that $\mathds{1}=\langle \{1\}\rangle$ and $\mathrm{std}=\langle\{x,y\}\rangle.$ Thus, we must have that $1,x,y$ are linearly independent and that $f(1)=\xty-\ytx$;
		$f(x)= \xtx+\xty-\ytx$; and $f(y)=\yty+\xty-\ytx.$ We shall show that this short exact sequence is non-split.\\ Suppose if possible there exists a $G_2\text{-linear}$ map $f':V\to W$ such that $f'\circ f=$Id. 
		Then we must have $f'(\xty-\ytx)=1.$\\
        Let $f'(\xty)=\alpha 1+\beta x+\gamma y$ for some scalars $\alpha,\beta,\gamma.$ Then, for $h=\left(\begin{matrix}
			0 & 1 \\
			1 & 0
		\end{matrix}\right),$ $f'(\ytx)=f'(h\cdot(\xty))=h\cdot f'(\xty)=h\cdot(\alpha 1+\beta x+\gamma y) = \alpha 1+\beta y+\gamma x.$ Also, $1=f'(\xty-\ytx)=(\alpha 1+\beta x+\gamma y)-(\alpha 1+\beta y+\gamma x)=(\beta-\gamma)(x-y),$ which contradicts that $1,x,y$ are linearly independent. Thus, there is no image of $\xty$ and $\ytx$ under any $G_2\text{-linear}$ map from $V$ to $W.$ Hence the above sequence is non-split. This proves that the socle of $V$ is $W.$ Since $V/W$ is $1$ dimensional, it is semisimple and thus we get the required socle filtration of $V.$
	\end{proof}
	
	\subsection{$V_{\vec{r}}(k)$ distinguished by characters of $\Gp$}\label{SS2}
	\begin{definition}
		Let $G$ and $H$ be groups with $H\leq G.$ Then $(G,H)$ is called a Gelfand pair if $\mathrm{dim}_{\OF}(\mathrm{Hom}_H(\pi,\mathds{1}))\leq 1,$ for all irreducible $\OF\text{-representations }\pi$ of $G.$
	\end{definition}

    \begin{definition}\label{Def}
		Let $G$ and $H$ be groups with $H\leq G$ and $\pi$ be a representation of $G.$ For a character $\chi$ of $H,$ we say that $\pi$ is $(H,\chi)\text{-distinguished}$ (or simply $H\text{-distinguished}$ if $\chi=\mathds{1}$) if $\mathrm{Hom}_H(\pi,\chi)\neq 0.$
	\end{definition}
	
	Let $0\leq r_1,r_2<p$ and $r_1+r_2<p.$ The author in \cite{RZ} has proved that $(\Gq, \Gp)$ is a Gelfand pair and so the dimension of $\text{Hom}_G(\pi,\mathds{1})$ is at most 1 for all irreducible representations $\pi$ of $\Gq.$ We explicitly prove when it is $0$ and $1$ for the irreducible representations $V_{\vec{r}}(k)$ of $\Gq$ with $r_1,r_2$ satisfying the above conditions. 
    
    \begin{proposition}
      Let $0\leq r_1,r_2<p$ and $r_1+r_2<p.$ Then $V_{\vec{r}}(k)$ is $\Gp\text{-distinguished}$ if and only if $r_1=r_2$ and $r_2\equiv -k(\text{mod } p-1).$
    \end{proposition}

    \begin{proof}
         Without loss of generality, assume that $r_1\geq r_2.$ By the (semisimple) decomposition in Remark \ref{R3}, an irreducible representation $V_{\vec{r}}(k)$ of $\Gq$ restricted to $\Gp$ contains a character if and only if $r_1=r_2.$ In particular, it contains the trivial character with multiplicity one if $r_2\equiv -k(\text{mod } p-1).$ Hence,
	$$\mathrm{dim}_{\OF}(\mathrm{Hom}_{G_p}(V_{\vec{r}}(k), \mathds{1}))
	=\left\{\begin{array}{ll}
		1, & \text{ if }r_1=r_2\text{ and }r_2\equiv -k(\text{mod } p-1),\\
		0, & \text{ otherwise.}
	\end{array}\right.$$
    \end{proof}

\begin{proposition}
    Let $\vec{r}=(r,r)$ and $r,k\geq 0.$ Then $V_{\vec{r}}(k)$ is $(\Gp,\mathrm{det}^{r+k})\text{-distinguished}.$
	
\end{proposition}
	
	\begin{proof}
	    We show the existence of a character of $\Gp$ in the socle (and consequently a character in the cosocle) of $V_{r}\otimes V_{r}\otimes\mathrm{det}^k$, for all $r,k\geq 0.$ To prove this, consider the product defined on the tensor product of polynomial algebra with homogeneous polynomials of varying degrees as its graded pieces \cite{SL}. For $f_1,f_2\in V_r$ and $g_1,g_2\in V_s,$ we have $$(f_1\otimes f_2)\cdot(g_1\otimes g_2)=(f_1g_1\otimes f_2g_2)\in V_{r+s}\otimes V_{r+s}.$$
	One can check that the product is well defined and hence can be extended linearly to the whole spaces $V_r\otimes V_r$ and $V_s\otimes V_s$ respectively through distributivity of this product over sums. Also, note that the product is commutative as the algebra of homogeneous polynomials is.
	
	Let $g=\left(\begin{matrix}
			a & b \\
			c & d
		\end{matrix}\right)\in \Gp$ and $\xty-\ytx\in V_1\otimes V_1.$ Then, one can easily check that $g\cdot(\xty-\ytx)=\mathrm{det}(g)(\xty-\ytx).$ So,	
		\begin{align*}		    			
			g\cdot(\xty-\ytx)^r
			&= \displaystyle{g\cdot\left(\sum\limits_{i=0}^r (-1)^i\binom{r}{i}(\xty)^{r-i}(\ytx)^i\right)} \\
			&=\displaystyle{g\cdot\left(\sum\limits_{i=0}^r (-1)^i\binom{r}{i}(X^{r-i}Y^i\otimes Y^{r-i}X^i)\right)}\\
			&= \displaystyle{\sum\limits_{i=0}^r (-1)^i\binom{r}{i}(aX+cY)^{r-i}(bX+dY)^i\otimes(bX+dY)^{r-i}(aX+cY)^i}\\
			&= \displaystyle{\sum\limits_{i=0}^r (-1)^i\binom{r}{i}((aX+cY)\otimes(bX+dY))^{r-i}((bX+dY)\otimes(aX+cY))^i}\\
			&= \displaystyle{[((aX+cY)\otimes(bX+dY))-((bX+dY)\otimes(aX+cY))]^r}\\
			&= [g\cdot (\xty-\ytx)]^r\\
			&= \mathrm{det}(g)^r(\xty-\ytx)^r.		
		\end{align*}

Thus, $\mathrm{det}^r$ is a subrepresentation of $V_r\otimes V_r$, and subsequently, $\mathrm{det}^{r+k}$ is that of $V_r\otimes V_r\otimes\mathrm{det}^k.$ Since taking contragredient is a contravariant functor, we have $$\left(V_r\otimes V_r\otimes\mathrm{det}^k\right)^*\twoheadrightarrow \left(\mathrm{det}^{r+k}\right)^*.$$
	Now, the dual of tensor product of two representations is isomorphic to the tensor product of their duals. Also, it is well known that $V_r(k)^*\cong V_r(-r-k).$ Thus, we have $V_r\otimes V_r\otimes\mathrm{det}^{-2r-k}\twoheadrightarrow \mathrm{det}^{-r-k}$ and so $V_r\otimes V_r\otimes \mathrm{det}^{k}\twoheadrightarrow \mathrm{det}^{r+k}.$ This proves that
    $$\mathrm{Hom}_{\Gp}(V_{\vec{r}}(k), \mathrm{det}^{r+k})\neq 0. $$
    \end{proof}
	
\section{Restriction when $\F_q$ is an extension of $\F_p$ of degree $3$}\label{S4}
	For this section, let $\F_q$ be an extension of $\F_p$ of degree $3.$ Here, the restriction of $V_{\vec{r}}(k)$ to $\Gp$  is equivalent to studying the representation $V_{r_1}\otimes V_{r_2}\otimes V_{r_3}\otimes\mathrm{det}^k$ of $\Gp.$ Again, we assume $k=0$ at first and similar to the quadratic case we will consider without loss of generality $0\leq r_3\leq r_2\leq r_1<p.$	To compactly state the results, we shall denote $\displaystyle{\bigoplus\limits_{i=0}^{s}a_iV_{f(i)}(i)\oplus\bigoplus\limits_{i=s+1}^{t}b_i}V_{f(i)}(i)$ by $\displaystyle{\left[\bigoplus\limits_{i=0}^{s}a_i\oplus\bigoplus\limits_{i=s+1}^{t}b_i\right]}V_{f(i)}(i),$ where $f(i)$ is some function of $i.$
	
	\begin{theorem}\label{T3}
		Let $\vec{r}=(r_1,r_2,r_3)$ such that $ 0\leq r_3\leq r_2\leq r_1< p-r_2-r_3.$ Then
		\begin{enumerate}
			\item if $r_1\geq r_2+r_3,$
			\begin{equation*}
			    {V_{\vec{r}}|}_{\Gp}\cong            	\displaystyle{\left[\bigoplus\limits_{i=0}^{r_3}(i+1)\oplus\bigoplus\limits_{i=r_3+1}^{r_2}(r_3+1)\oplus \bigoplus\limits_{i=r_2+1}^{r_2+r_3}(r_3+(r_2+1)-i)\right]V_{r_1+r_2+r_3-2i}(i),}
			\end{equation*}				
			
			\item if $r_1< r_2+r_3,$
			\begin{multline*}
			    {V_{\vec{r}}|}_{\Gp}\cong 
				\left[\displaystyle{\bigoplus\limits_{i=0}^{r_3}(i+1)\oplus\bigoplus\limits_{i=r_3+1}^{r_2}(r_3+1)\oplus\bigoplus\limits_{i=r_2+1}^{r_1}(r_3+(r_2+1)-i)\oplus}\right.\\
\left.\displaystyle{\bigoplus\limits_{i=r_1+1}^{\left\lfloor\frac{r_1+r_2+r_3}{2}\right\rfloor}(r_1+r_2+r_3-2i+1)}\right] V_{r_1+r_2+r_3-2i}(i),
			\end{multline*}
where $\lfloor \cdot\rfloor$ denotes the floor function.
			
		\end{enumerate}
	\end{theorem}
	
	\begin{proof}
		We will prove this theorem by induction. Note that for $r_3=0$, $V_{r_1}\otimes V_{r_2}\otimes V_0\cong V_{r_1}\otimes V_{r_2}$ and so by Theorem \ref{T2}, the result is true in this case. With the induction hypothesis that the theorem holds for $r_3-1\leq r_2-1\leq r_1$, we will show that it holds for $r_3\leq r_2\leq r_1$. (So to decompose $V_{r_1}\otimes V_{r_2}\otimes V_{r_3},$ we can apply induction $r_3$ times on the initial case $V_{r_1}\otimes V_{r_2-r_3}\otimes V_{0}.$) Note that,
		\begin{align*}
			V_{r_1}\otimes V_{r_2}\otimes V_{r_3}&\cong V_{r_1}\otimes \left(\displaystyle{\bigoplus\limits_{i=0}^{r_3}V_{r_2+r_3-2i}(i)}\right)~~~~~(\text{By Theorem }\ref{T2})\\
			&\cong (V_{r_1}\otimes V_{r_2+r_3})\oplus (V_{r_1}\otimes (V_{r_2-1}\otimes V_{r_3-1})\otimes \mathrm{det}).
		\end{align*}\\
        
		Using induction hypothesis and Theorem \ref{T2},
		\begin{enumerate}
			\item when $r_1\geq r_2+r_3,$            
			\begin{multline*}			    
			V_{r_1}\otimes V_{r_2}\otimes V_{r_3}\cong \displaystyle{\bigoplus\limits_{i=0}^{r_2+r_3}V_{r_1+r_2+r_3-2i}(i)\oplus}\left[\left(\displaystyle{\bigoplus\limits_{i=0}^{r_3-1}(i+1)\oplus \bigoplus\limits_{i=r_3}^{r_2-1}(r_3)\oplus}\right.\right.\\
\left.\left.\displaystyle{\bigoplus\limits_{i=r_2}^{r_2+r_3-2}(r_3-1+(r_2)-i)}\right)V_{r_1+r_2-1+r_3-1-2i}(i)\otimes \mathrm{det}\right]		
            \end{multline*}
			$\displaystyle{\cong\bigoplus\limits_{i=0}^{r_2+r_3}V_{r_1+r_2+r_3-2i}(i)\oplus\left[\left(\bigoplus\limits_{i=1}^{r_3}i\oplus \bigoplus\limits_{i=r_3+1}^{r_2}r_3\oplus \bigoplus\limits_{i=r_2+1}^{r_2+r_3-1}(r_3+r_2-i)\right)V_{r_1+r_2+r_3-2i}(i)\right]}$\\			
$\displaystyle{\cong \left(\bigoplus\limits_{i=0}^{r_3}(i+1)\right.}\oplus\displaystyle{\left.\bigoplus\limits_{i=r_3+1}^{r_2}(r_3+1)\oplus\bigoplus\limits_{i=r_2+1}^{r_2+r_3}(r_3+(r_2+1)-i)\right)}V_{r_1+r_2+r_3-2i}(i),$\\

			\item when $r_1<r_2+r_3,$
			\begin{multline*}
			    V_{r_1}\otimes V_{r_2}\otimes V_{r_3}\cong \displaystyle{\bigoplus\limits_{i=0}^{r_1}V_{r_1+r_2+r_3-2i}(i)\oplus}\left[\mathrm{det}\otimes\left(\displaystyle{\bigoplus\limits_{i=0}^{r_3-1}(i+1)\oplus \bigoplus\limits_{i=r_3}^{r_2-1}r_3\oplus\bigoplus\limits_{i=r_2}^{r_1}(r_3-1+(r_2)-i)\oplus}\right.\right.\\		
\displaystyle{\left.\left.\bigoplus\limits_{i=r_1+1}^{\left\lfloor\frac{r_1+r_2+r_3-2}{2}\right\rfloor}(r_1+r_2+r_3-2i-1)\right)V_{r_1+r_2+r_3-2i-2}(i)\right]}			\end{multline*}
            \begin{multline*}
			    \displaystyle{\cong \bigoplus\limits_{i=0}^{r_1}V_{r_1+r_2+r_3-2i}(i)\oplus \left[\left(\bigoplus\limits_{i=1}^{r_3}i\oplus \bigoplus\limits_{i=r_3+1}^{r_2}r_3\oplus\bigoplus\limits_{i=r_2+1}^{r_1+1}(r_3+r_2-i)\oplus\right.\right.}\\
\displaystyle{\left.\left.\bigoplus\limits_{i=r_1+2}^{\left\lfloor\frac{r_1+r_2+r_3}{2}\right\rfloor}(r_1+r_2+r_3-2i+1) \right)V_{r_1+r_2+r_3-2i}(i)\right]}            
			\end{multline*}
            \begin{equation*}
        \displaystyle{\cong\left(\bigoplus\limits_{i=0}^{r_3}(i+1)\oplus\bigoplus\limits_{i=r_3+1}^{r_2}(r_3+1)\oplus\bigoplus\limits_{i=r_2+1}^{r_1}(r_3+(r_2+1)-i)\oplus\bigoplus\limits_{i=r_1+1}^{\left\lfloor\frac{r_1+r_2+r_3}{2}\right\rfloor}(r_1+r_2+r_3-2i+1)\right)}V_{r_1+r_2+r_3-2i}(i).
            \end{equation*}
		\end{enumerate}
		This proves the theorem.
	\end{proof}
	
	\begin{remark}
		With the notations used above, for $k\in \{0,1,\cdots, q-2\}$ and $0\leq r_3\leq r_2\leq r_1<p-r_2-r_3,$
		$${V_{\vec{r}}(k)|}_{\Gp}\cong ~\bigoplus\limits_{i=0}^{r_2+r_3}\alpha_iV_{r_1+r_2+r_3-2i}((i+k))$$ where $\alpha_i$ is the multiplicity of $V_{r_1+r_2+r_3-2i}(i)$ in the decomposition of Theorem \ref{T3}.
	\end{remark}
	
	\begin{corollary}
		Let $0\leq r_3\leq r_2\leq r_1<p-r_2-r_3$ and $W_i$ denote $V_{r_1+r_2+r_3-2i}.$ Then we have the following $\mathrm{SL}_2(\F_p)\text{-isomorphisms}.$
\begin{enumerate}
			\item If $r_1\geq r_2+r_3,$ then
			\begin{equation*}
			    \displaystyle{\left(V_{r_1}\otimes V_{r_2}\otimes V_{r_3}\right)\cong
				\left[\bigoplus\limits_{i=0}^{r_3}(i+1)\oplus\bigoplus\limits_{i=r_3+1}^{r_2}(r_3+1)\oplus\bigoplus\limits_{i=r_2+1}^{r_2+r_3}(r_3+(r_2+1)-i)\right]W_i.}\end{equation*}

			\item If $r_1< r_2+r_3,$ then			
            \begin{multline*}
            \left(V_{r_1}\otimes V_{r_2}\otimes V_{r_3}\right)\cong
			\displaystyle{\left[\bigoplus\limits_{i=0}^{r_3}(i+1)\oplus\bigoplus\limits_{i=r_3+1}^{r_2}(r_3+1)\oplus\bigoplus\limits_{i=r_2+1}^{r_1}(r_3+(r_2+1)-i)\oplus\right.}\\
\displaystyle{\left.\bigoplus\limits_{i=r_1+1}^{\left\lfloor\frac{r_1+r_2+r_3}{2}\right\rfloor}(r_1+r_2+r_3-2i+1)\right]W_i.}
			    \end{multline*}
			
		\end{enumerate}
	\end{corollary}
	
	\begin{proof}
		The proof follows from Theorem \ref{T3}: the $\Gp\text{-isomorphism}$ in the theorem implies $\text{SL}_2(\F_p)\text{-isomorphism}$ and the twists by determinant powers are trivial.
	\end{proof}
	
	\section{Conclusion}
	The restriction of irreducible representations of $\Gq$ to $\Gp$ translated into finding the decomposition (in a broad sense) of tensor product of irreducible representations of $\Gp.$ Evidently, the decomposition is only dependent on the tuple $\Vec{r}$ and the prime $p.$ When $p=2,$ we have a complete description of the restriction in the case $q=p^2.$ Also, we have decomposed the irreducible representations of $\Gq$ for all $(r_1,r_2,r_3)$ such that $r_1+r_2+r_3<p,$ when $q=p^3.$ The decomposition of all the other cases that is, $p\leq r_1+r_2\leq 2(p-1)$ when $q=p^2$ and $p\leq r_1+r_2+r_3\leq 3(p-1)$ when $q=p^3$ still remain unanswered. More generally, when $\F_q$ is an extension of $\F_p$ of degree $n,$ we can inductively use the decomposition of tensor product of two irreducible representations and the distributivity of tensor product to conclude the decomposition for appropriate ranges of $r_i$'s in $\Vec{r}.$ But the complicated calculations in the case $q=p^3$ were evidence of the fact that simplification of the decomposition will not be that straightforward.
	
	There was an interesting development along the same lines. When $\Vec{r}=(r,r),$ we have at least one decomposition of $V_r\otimes V_r$ namely, $\text{Sym}^2(V_r)\oplus \text{Alt}^2(V_r).$ The $1\text{-dimensional}$ subrepresentation computed in Subsection \ref{SS2} is in fact a subrepresentation of $\text{Sym}^2(V_r)$ when $r$ is even and that of $\text{Alt}^2(V_r)$ when $r$ is odd. Through explicit computations for initial cases, we conjecture that the following should hold for $0\leq r<p-1,$
	$$\text{Sym}^2(V_r)\otimes\mathrm{det}\cong \text{Alt}^2(V_{r+1}).$$
	Obviously, the dimension of both the representations is $\frac{(r+1)(r+2)}{2}.$ Further, the ordinary characters of both the representations coincide but we can not say for a fact if they will be isomorphic, even up to semisimplification.\\
    
 \vspace{5mm}
        {\noindent \bf Acknowledgements:} The author is grateful to Prof. Gautam Borisagar for suggesting the question and numerous helpful discussions. The author is also thankful to Prof. Dipendra Prasad for his fruitful suggestions during the preparation of this article. The author also acknowledges the Government of Gujarat for the financial support through SHODH (ScHeme Of Developing High Quality Research) fellowship.


\begin{thebibliography}{20}

 \bibitem{DB}{\sc D. Benson}, {\it Modular Representation Theory}, Notes by Kay Jin Lim and Sejong Park, Autumn session, Aberdeen (2007).
		\bibitem{DG} {\sc D.~J. Glover}, A study of certain modular representations, {\it J. Algebra} {\bf 51 (2)} (1978), 425--475.
		\bibitem{JPS} {\sc J.~P. Serre}, Sur la semi-simplicit\'e{} des produits tensoriels de repr\'esentations de groupes, {\it Invent. Math.} {\bf 116 (1-3)} (1994), 513--530.
        \bibitem{KC} {\sc K.~Y. Chan}, Quotient branching law for $p$-adic $(\mathrm{GL}_{n+1}, \mathrm{GL}_n)$ I: generalized Gan-Gross-Prasad relevant pairs, {\it Preprint} https://arxiv.org/pdf/2212.05919
		\bibitem{BL} {\sc L. Barthel, R.~A. Livn\'e}, Irreducible modular representations of ${\rm GL}_2$ of a local field, {\it Duke Math. J.} {\bf 75 (2)} (1994), 261--292.
        \bibitem{CLL} {\sc P. Cui, T. Lanard, H. Lu}, Modulo $\ell$ distinction problems, {\it Compos. Math.} {\bf 160 (10)} (2024), 2285--2321.

		\bibitem{RZ}{\sc R. Zhang}, Modular Gelfand pairs and multiplicity-free representations, {\it Int. Math. Res. Not. IMRN} {\bf 7} (2024), 5490--5523.
		\bibitem{SL}{\sc S. Lang}, {\it Algebra}, revised third edition, Graduate Texts in Mathematics, 211, Springer, New York (2002).
		\bibitem{FH}{\sc W. Fulton, J.~D. Harris}, {\it Representation theory}, Graduate Texts in Mathematics Readings in Mathematics, 129 , Springer, New York (1991).
		
        \bibitem{VS}{\sc V. S\'echerre, C.~G. Venketasubramanian}, Modular representations of ${\rm GL}(n)$ distinguished by ${\rm GL}(n-1)$ over a $p$-adic field, {\it Int. Math. Res. Not. IMRN} {\bf 14} (2017), 4435--4492.

\end{thebibliography}
\end{document}